\documentclass[12pt]{amsart}
\usepackage{amsmath}
\usepackage{amsthm}
\usepackage{amssymb}
\usepackage{amsfonts,mathrsfs}
\usepackage{stmaryrd}
\usepackage{amsxtra}  
\usepackage{epsfig}
\usepackage{verbatim}
\usepackage[all]{xy}

\theoremstyle{plain}
\newtheorem{theorem}[equation]{Theorem}
\newtheorem{proposition}[equation]{Proposition}
\newtheorem{lemma}[equation]{Lemma}
\newtheorem{corollary}[equation]{Corollary}
\newtheorem{definition}[equation]{Definition}
\theoremstyle{remark}
\newtheorem{remark}[equation]{Remark}
\numberwithin{equation}{section}

\newcommand{\sjump}{\hskip .2 cm}

\newcommand{\dbarstar}{\bar{\partial}^{\star}}

\newcommand{\psum}{\sideset{}{^{\prime}}{\sum}}
\newcommand{\dbar}{\bar \partial}

\newcommand{\re}{\text{Re}}
\newcommand{\im}{\text{Im}}

\newcommand{\sd}{{\mathscr D}}

\newcommand{\sn}{{\mathscr N}}

\newcommand{\sr}{{\mathscr R}}

\newcommand{\sx}{{\mathscr X}}

\newcommand{\C}{{\mathbb C}}

\newcommand{\R}{{\mathbb R}}

\begin{document}

\title[Closed range]{On closed range for $\dbar$}
\author{A.-K. Herbig  \& J. D. McNeal}
\subjclass[2010]{32W05}
\begin{abstract}
A sufficient condition for $\dbar$ to have closed range is given for pseudoconvex, possibly unbounded domains in $\mathbb{C}^{n}$.

\end{abstract}
\thanks{Research of the first author was partially supported by a Austrian Science Fund FWF grant and a Grant-in-Aid for Scientific Research of the JSPS}
\thanks{Research of the second author was partially supported by a National Science Foundation grant.}
\address{Graduate School of Mathematics, \newline Nagoya University, Nagoya, Japan}
\email{herbig@math.nagoya-u.ac.jp}
\address{Department of Mathematics, \newline The Ohio State University, Columbus, Ohio, USA}
\email{mcneal@math.ohio-state.edu}

\maketitle 


\section{Introduction}\label{S:intro}

Extending the Cauchy--Riemann operator, $\dbar$, initially defined pointwise, to an unbounded operator on $L^2$ allows Hilbert 
space methods to bear on existence and regularity
questions connected to the Cauchy--Riemann equations. These methods allow one to deduce powerful results about complex 
function theory, especially in several complex variables. 
Results obtained in this manner, after the seminal work of Kohn and H\" ormander in the early 1960s, are perhaps well-known enough to view extending $\dbar$ to $L^2$ as
a classical part of complex analysis.

A basic question, underlying more refined existence and regularity issues, is whether the extended $\dbar$ operator 
has closed range in $L^2$. In this paper, for $\Omega\subset\mathbb{C}^n$ a pseudoconvex domain, we give a general sufficient 
condition for $\dbar$ to have closed range in $L^2_{p,q}(\Omega)$. This condition is not restricted to bounded domains; indeed, 
this paper primarily grew out of our interest in determining classes of unbounded domains for which $\dbar$ has closed 
range.  A secondary interest was understanding the closed range property on non-smooth domains.
The condition given is sensitive to the bi-degree $(p,q)$ where the 
question of closed range is posed. 

We do not consider non-pseudoconvex domains in this paper. For the few results on closed range for $\dbar$ on (classes of)
non-pseudoconvex domains, see  \cite{Ho95,Shaw10,ChakrabartiShaw11,LiShaw13}. The theory needs more results on general domains,
of both positive and negative type.

If $\Omega$ is a bounded pseudoconvex domain in $\mathbb{C}^n$, the fact that $\dbar$ has closed range in $L^2$, 
in all bi-degrees, follows from H{\"o}rmander's estimates and the fact that $\Omega$ supports a  bounded
uniformly strictly plurisubharmonic function, e.g., $\phi(z)=|z|^2$. See Theorem 2.2.1 in \cite{Hormander65} for the essential 
inequality; that closed range of $\dbar$ follows from this can be achieved by exhausting the domain using Theorem \ref{T:approximatingsd} and arguing as
in Proposition \ref{P:bdduniformest} below. Our sufficient condition is also potential-theoretic, but more general than supporting a 
bounded function like $\phi$. The condition requires $\Omega$ to support two functions whose 
first and second derivatives combine in a certain fashion to give a uniformly positive lower bound. That this condition implies
$\dbar$ has closed range follows from the refined, twisted $\dbar$ estimates in \cite{McNeal02} rather than H{\"o}rmander's 
estimates.
Two special cases of the general condition are also given, removing the interplay between two functions on $\Omega$. The 
simpler hypotheses in Corollaries \ref{C:bounded} and \ref{C:self-bounded} are often adequate for determining when $\dbar$ has 
closed range in practice.
The examples discussed in Section \ref{S:examples} use only Corollary \ref{C:bounded1}.

The main theorem guaranteeing closed range is stated in Section \ref{S:bounded} for bounded domains with smooth boundary, but attention is paid to the size of the constant obtained in order to pass to the unbounded, non-smooth cases. After reviewing how an arbitrary pseudoconvex
domain can be approximated by smoothly bounded pseudoconvex open sets in Section \ref{S:approximatingsd}, it is shown in Section \ref{S:unbounded} how a uniform version of the closed range inequality \eqref{E:standard} on the approximating subsets implies that $\dbar$ has closed range 
on the limit domain.

In Section \ref{S:examples}, we give examples of unbounded pseudoconvex domains where $\dbar$ has closed range and others 
where it does not. For domains $\Omega$ in the plane, it is natural to conjecture that $\dbar$ has closed range if and only if 
$\Omega$ does not contain
arbitrarily large complex discs. We show that this condition is necessary in general, but only establish sufficiency with an 
additional hypothesis (see Definition \ref{D:largedisks}). In higher dimensions, it is reasonable to expect that if $\dbar$ on $(0,q)$-
forms has closed range, then
$\Omega$ cannot contain arbitrarily large $q+1$-dimensional Euclidean balls. A sufficient condition for closed range will likely 
involve holomorphic images of balls, though the relation between the form level $\dbar$ acts on and the dimension of the images 
should remain. We hope to
return to this matter in another paper.

The authors thank the referee for spotting a gap in the proof of Lemma \ref{L:closedrangeapprox}, and for several suggestions that
improved the exposition of the paper.

\medskip


\section{Preliminaries}\label{S:prelim}

Let $\Omega$ be a domain in $\mathbb{C}^n$. We write $\Omega\Subset\mathbb{C}^n$ to indicate $\Omega$ is bounded; more generally, write $U\Subset V$, for $U,V$ open sets, 
to indicate that $\overline U$ is a compact subset of $V$.
Whether bounded or not, we shall say $\Omega$ has smooth boundary if there is a
smooth, real-valued function $r$ such that $\Omega=\{ r<0\}$ and $dr\neq 0$ when $r=0$; $r$ is then a defining function for 
$\Omega$.

For $1\leq q\leq n$, a $(0,q)$-form $u$ can be uniquely written
\begin{align}\label{E:formdef}
  u(z)=\psum_{|I|=q} u_I(z)\, d\bar z^I,
\end{align}  
where $\sum'_{|I|=q}$ denotes the sum over increasing multi-indices
$I$ of length $q$, $u_I(z)$ are functions, and $d\bar z^I = d\bar z_{i_1}\wedge\dots\wedge d\bar z_{i_{q}}$ 
when $I=\left(i_1,\dots , i_{q}\right)$ is such a multi-index.

If $v=\sum'_{|I|=q} v_I\, d\bar z^I$ is another $(0,q)$-form, define the inner product
\begin{align}\label{E:innerprod}
(u,v)_{L^{2}_{0,q}(\Omega)}= \psum_{|I|=q}\int_\Omega u_I(z)\overline{v_I(z)}\, dV_E,
\end{align}
where $dV_E$ is the Euclidean volume element. Only products of components of $u$ and $v$ corresponding to the same 
multi-index $I$ appear in the integrand in \eqref{E:innerprod}; in particular, forms of different bi-degree are orthogonal. 
Let $L^2_{0,q}(\Omega)$
denote the $(0,q)$-forms $u$ on $\Omega$ such that $\|u\|_{L^{2}_{0,q}(\Omega)}<\infty$, where $\|.\|_{L^{2}_{0,q}(\Omega)}$ is the norm induced by the inner product \eqref{E:innerprod}. The subscripts will be dropped when confusion is unlikely. To distinguish between the $L^{2}$-norms and the Euclidean norm on $\mathbb{C}^{n}$ we use the notation $|.|$ for the latter.

Let $\Lambda^{0,q}(\Omega)$, $\Lambda^{0,q}_{c}(\Omega)$, and $\Lambda^{0,q}(\overline{\Omega})$ be the space of 
$(0,q)$-forms with coefficients in $\mathcal{C}^{\infty}(\Omega)$, $\mathcal{C}^{\infty}_{c}(\Omega)$, and 
$\mathcal{C}^{\infty}(\overline{\Omega})$, respectively. Denote the domain, range, and null space of an operator $A$ by 
$\sd(A), \sr(A),$ and $\sn(A)$, respectively.

\medskip

The Cauchy--Riemann operator, $\dbar$, on functions $f\in\mathcal{C}^{\infty}(\Omega)$ is defined as
$$\dbar f:=\sum_{j=1}^{n}\frac{\partial f}{\partial\bar{z}_{j}}d\bar{z}_{j}.$$
It is extended to $(0,q)$-forms by linearity, 
$$\dbar_{q}u:=\sum_{j=1}^{n}\psum_{|I|=q}\frac{\partial u_{I}}{\partial\bar{z}_{j}}d\bar{z}_{j}\wedge d\bar{z}^{I},$$
where $u\in\Lambda^{0,q}(\Omega)$ is given by \eqref{E:formdef}.

$\dbar_{q}$ is extended to an $L^{2}$-operator (still called $\dbar_{q}$) by first letting it act on $L^{2}_{0,q}(\Omega)$ in the 
sense of distributions and then restricting its domain, $\sd(\dbar_{q})$, as follows:
\begin{align*}
  \sd(\dbar_{q})=\left\{u\in L^{2}_{0,q}(\Omega): \dbar_{q}u\in L^{2}_{0,q+1}(\Omega)\right\}.
\end{align*}
This is the maximal extension of $\dbar_{q}$ to $L^2_{0,q}(\Omega)$. An equivalent description of $\sd(\dbar_{q})$ is
\begin{align}\label{D:domaindbar2}
   \sd(\dbar_{q})=\left\{u\in L^2_{0,q}(\Omega):\; \right.&\exists\;\{u_j\}\subset\Lambda^{0,q}\left(\overline\Omega\right)
   \text{such that } u_j
   \to u \\ &\text{in } L^2_{0,q}(\Omega)
   \left.\text{and} \left\{\dbar_q u_j\right\}\text{is Cauchy in } L^2_{0,q+1}(\Omega)\right\}.\notag
\end{align}
Then one sets $\dbar_q u=\lim_{j\to\infty} \dbar_q u_j$ and checks easily that this is independent of the sequence $\{u_j\}$. The extended operator $\dbar_q$ is closed
and densely defined on $L^2_{0,q}(\Omega)$.

The Hilbert space adjoint, $\dbarstar_{q+1}$, of $\dbar_{q}$ is the operator with domain
\begin{align}\label{D:domaindbarstar}
  \sd(\dbarstar_{q+1})=\left\{v\in L^{2}_{0,q+1}(\Omega): \exists\;C>0\text{ with }|(\dbar_{q} u,v)|\leq C\|u\|\sjump\forall u\in 
  \sd(\dbar_{q})\right\}
\end{align}
satisfying for $v\in\sd(\dbarstar_{q+1})$ $$(\dbar_{q}u,v)=(u,\dbarstar_{q+1}v)\qquad\forall\sjump u\in\sd(\dbar_{q}).$$
The subspace $\mathcal{D}^{0,q}(\Omega) := \sd\left(\dbarstar_q\right)\cap\Lambda^{0,q}\left(\overline\Omega\right)$ is useful for computations. The abstract conditions for $u$ to belong to $\sd(\dbarstar_{q})$ become explicit
boundary conditions if $u\in\Lambda^{0,q}\left(\overline\Omega\right)$. Also, if $\Omega$ is bounded and has smooth boundary, then $\mathcal{D}^{0,q}(\Omega)$ is dense in $\sd\left(\dbarstar_q\right)\cap\sd\left(\dbar_q\right)$;
see  \cite{Hormander65}, pages 94--98.

\medskip

Finally, to use shorthand to denote the action of a complex Hessian on a $(0,q)$-form the following notation is introduced. For any $1\leq m\leq n$ and $H$ an increasing index of length $q-1$, let $mH$ denote the multi-index $m,h_{i_{1}},\dots,h_{i_{q-1}}$ and, if $m\notin H$, $\langle mH\rangle$ the increasing multi-index formed from the set $\{ m, h_{i_1},\dots , h_{i_{q-1}}\}$. For $u$  given by \eqref{E:formdef}, set
$$u_{mH} =\epsilon^{mH}_{\langle mH\rangle}\,  u_{\langle mH\rangle},$$
where 
\begin{align*}
      \epsilon^{mH}_{\langle mH\rangle}=\left\{
       \begin{array}{cl}
               \text{sign of permutation turning } \\ 
               mH\text{ into } \langle mH\rangle&\qquad\text{if } m\notin H\\
                                                           \\
               0&\qquad\text{if } m\in H
        \end{array} .\right.
\end{align*}
If $f$ is a $\mathcal{C}^2$ function, define
\begin{align}\label{E:hess_q}
i\partial\dbar f(u,u):=\psum_{|J|=q-1}\sum_{k,l=1}^{n}\frac{\partial^2 f}{\partial z_{l}\partial\bar{z}_{k}}u_{lJ}
\bar{u}_{kJ},\qquad u\in\Lambda^{0,q}(\Omega).
\end{align}
When $q=1$, \eqref{E:hess_q} is standard, expressing the natural action of the $(1,1)$-form $i\partial\dbar f$ on the {\it vectors} 
$u$ and $\bar u$ associated to the forms $u,\bar u$ by the Euclidean structure of $\C^n$. For example, a  $\mathcal{C}^{2}$ function $f$ is plurisubharmonic on an open $U\subset\C^n$ if 
$i\partial\dbar f(u,u)(p)\geq 0$ for all $u\in\Lambda^{0,1}(\Omega)$ and $p\in U$.

For $q>1$, the right-hand side of 
\eqref{E:hess_q} is less natural. But this
expression arises repeatedly when integrating by parts in the $\dbar$-complex and representing it by the left-hand side of 
\eqref{E:hess_q} shortens many formulas, e.g., \eqref{E:basictwistedest} below.

In the sequel, we shall use two equivalent notions of pseudoconvexity. If $\Omega\subset\C^n$ is an arbitrary domain, we say that $\Omega$ is pseudoconvex if there 
exists a plurisubharmonic exhaustion function, $\Phi$, on $\Omega$, i.e., a plurisubharmonic $\Phi$ such that
\begin{equation}\label{D:psc}
\left\{z:\Phi(z) <c\right\} \Subset \Omega\qquad\forall c\in\R.
\end{equation}
It may be assumed that this exhaustion function is in fact smooth on $\Omega$, see, e.g.,  Theorem 2.6.11 in \cite{Hormander90}.
A domain $\Omega=\{z\in\mathbb{C}^{n}:r(z)<0\}$ with  smooth boundary, is said to be Levi pseudoconvex if
\begin{equation*}\label{D:Lpsc}
  i\partial\dbar r(\xi,\xi)(p)\geq 0\qquad p\in b\Omega,\,\, 
  \xi\in\mathbb{C}^{n}\;\;\text{with}\;\;\sum_{j=1}^{n}\frac{\partial r}{\partial z_{j}}(p)\xi_{j}=0.
\end{equation*}
A proof that these two notions are equivalent for open, smoothly bounded sets is given in Theorem 2.6.12 of \cite{Hormander90}.

\medskip


\section{Functional analysis}\label{S:funcanalysis}

The range of $\dbar_{q}$ is said to be closed if $\sr\left(\dbar_q\right)\subset L^2_{0,q+1}(\Omega)$ is metrically closed in $ L^2_{0,q+1}(\Omega)$. The closedness of $\sr\left(\dbar_q\right)$ is equivalent to $\dbar_q$ being norm-bounded
from below off its null space, $\sn\left(\dbar_q\right)$, and also to estimates from below on $\dbarstar_{q+1}$.
The following result summarizes these facts and is well-known (see, e.g., Theorem 1.1.1. in \cite{Hormander65}):

\begin{proposition}\label{P:standard}
The following conditions are equivalent. 
\begin{itemize} 
   \item[(i)] $\sr\left(\dbar_q\right)$ is closed in $L^{2}_{0,q+1}(\Omega)$.
   \item[(ii)] There exists a constant $C>0$ such that
        \begin{align}\label{E:standard}
           \left\| u\right\|\leq C\left\|\dbar_{q}u\right\|\qquad\forall 
           \sjump u\in \sd\left(\dbar_q\right)\cap \sn\left(\dbar_q\right)^\perp.
       \end{align}
  \item[(iii)] There exists a constant $C>0$ such that
        \begin{align*}
           \left\| v\right\|\leq C\left\|\dbarstar_{q+1}v\right\|
           \qquad\forall\sjump v\in \sd\left(\dbarstar_{q+1}\right)\cap \sn\left(\dbarstar_{q+1}\right)^\perp.
       \end{align*}
  \end{itemize} 
\end{proposition} 

\medskip

A slightly more flexible, but equivalent, inequality is also of interest.

\begin{proposition}\label{P:flexible}
$\sr\left(\dbar_q\right)$ is closed if and only if there exists $C>0$ such that
\begin{equation}\label{closed_2}
\text{dist}\left( v,\sn\left(\dbarstar_{q+1}\right)\right)\leq C\left\|\dbarstar_{q+1} v\right\|\qquad\forall\sjump v\in \sd\left(\dbarstar_{q+1}\right).
\end{equation}
\end{proposition}

\begin{proof} 
   Assume (iii) of Proposition \ref{P:standard} holds. Let $v=a\oplus b$, where $b\in\sn\left(\dbarstar_{q+1}\right)$ 
   and $a\in\sn\left(\dbarstar_{q+1}\right)^\perp$. 
   Then $\left\|\dbarstar_{q+1} v\right\|=\left\|\dbarstar_{q+1} a\right\|$
   and $\|a\|= \text{dist}\left(v,\sn\left(\dbarstar_{q+1}\right)\right)$. Thus,
   \begin{align*}
      \left\|\dbarstar_{q+1}v\right\|=\left\|\dbarstar_{q+1}a\right\| &\geq \frac 1C \left\|a\right\| \\
     & =\frac 1C\, \text{dist}\left(v,\sn\left(\dbarstar_{q+1}\right)\right),
   \end{align*}
   and so \eqref{closed_2} holds. That \eqref{closed_2} implies  (iii) of Proposition \ref{P:standard} is trivial.
\end{proof}

\medskip

Closed range properties of the $\dbar$-operator are closely connected to the existence of the 
$\dbar$-Neumann operator. For $1\leq q\leq n$, the $\dbar$-Neumann operator, $N_{q}$, is the 
solution operator to the following problem: given $\alpha\in L^{2}_{0,q}(\Omega)$, find 
$u\in L^{2}_{0,q}(\Omega)$ such that 
\begin{align*}\Box_{q}u:=&\left(\dbar_{q-1}\dbarstar_{q}+\dbarstar_{q+1}\dbar_{q}\right)u=\alpha, \text{ and} \\
u\in\sd(\Box_{q}):=&\left\{u\in\sd(\dbar_{q})\cap\sd(\dbarstar_{q}):\dbar_{q}u\in\sd(\dbarstar_{q+1}),\;\dbarstar_{q}u\in
\sd(\dbar_{q-1}) \right\}. 
\end{align*}
The relationship between  $L^{2}$-boundedness of $N_{q}$ and the closed range property for $\dbar$ is:
\begin{proposition}
    Let $\Omega\Subset\mathbb{C}^{n}$ be a pseudoconvex domain with smooth boundary, $1\leq q\leq n$. Then both $\dbar_{q-1}$ 
    and $\dbar_{q}$ have closed range in $L^{2}_{0,q}(\Omega)$ and $L^{2}_{0,q+1}(\Omega)$, respectively, 
    if and only if $N_{q}$ 
    is a bounded operator on $L_{0,q}^{2}(\Omega)$.
\end{proposition}
\begin{proof}
  This is fairly standard, so we only sketch the proof. It is straightforward to show that both $\dbar_{q-1}$ and $\dbar_{q}$ have closed range if and only if there is a constant $C>0$ 
  such that
  \begin{align}\label{E:basicestN}
    \|u\|\leq C\left(\|\dbar_{q}u\|+\|\dbarstar_{q}u\|\right)\sjump\qquad\forall\sjump u\in\sd(\dbar_{q})\cap\sd(\dbarstar_{q})
    \cap\left(\sn(\dbar_{q})\cap\sn(\dbarstar_{q})\right)^{\perp}
  \end{align}
  holds, see for instance Theorem 1.1.2 in \cite{Hormander65}. Moreover, \eqref{E:basicestN} implies that
  \begin{align*}
    \|u\|\leq C\|\Box_{q}u\|\sjump\qquad\forall\sjump u\in\sd(\Box_{q})\cap\sn(\Box_{q})^{\perp}.
  \end{align*}
  It follows that $\Box_{q}$ has closed range, since it is a closed operator, which yields the Hodge decomposition
  $L^{2}_{0,q}(\Omega)=\sn(\Box_{q})\oplus\sr(\Box_{q})$. 
  
  However, pseudoconvexity of $\Omega$ forces $\sn(\Box_{q})=\{0\}$. This follows, for instance, from \eqref{E:basictwistedest} below, with $\lambda=0$ and $\tau= B-|z|^2$ for a suitably large constant $B>0$.
  Hence, $\Box_{q}:\sd(\Box_{q})\rightarrow L^{2}_{0,q}(\Omega)$ is bijective and has a bounded inverse, $N_{q}$.
  
 To show that \eqref{E:basicestN} follows if $N_{q}$ is a bounded operator, one first shows that both $\dbar_{q}N_{q}$ and $\dbarstar_{q}N_{q}$ are bounded operators; this fact follows since $\|\dbar_{q}N_{q}u\|^{2}+\|\dbarstar_{q}N_{q}u\|^{2}=(u,Nu)$. Then for $u\in\sd(\dbar_{q})\cap\sd(\dbarstar_{q})$ 
 \begin{align*}
   \|u\|^{2}=(u,u)&=\left(\Box_{q}N_{q}u,u\right)=\left(\dbar_{q}N_{q}u,\dbar_{q}u\right)
   +\left(\dbarstar_{q}N_{q}u,\dbarstar_{q}u\right)\\
   &\leq\left\|\dbar_{q}N_{q}\right\|\cdot\left\|\dbar_{q}u\right\|+\left\|\dbarstar_{q}N_{q}u\right\|\cdot\left\|\dbarstar_{q}u\right\| \\
  &\leq C\|u\|\left(\left\|\dbar_{q}u\right\|+\left\|\dbarstar_{q}u\right\|\right),
 \end{align*}
 which yields  \eqref{E:basicestN}.
\end{proof}

\medskip

\section{Smoothly bounded domains; uniform estimates}\label{S:bounded}

The twisted estimates derived in Proposition 3.2 in \cite{McNeal02} (with $g=\tau$ and $\nu =1$) yield the following.

\begin{proposition}\label{P:twistedestimate}
Let  $\Omega$ be a  bounded, pseudoconvex domain in $\mathbb{C}^{n}$ with smooth boundary, $0\leq q\leq n-1$.  Let $\lambda,\tau\in\mathcal{C}^{2}(\overline{\Omega})$ and $\tau\geq 0$.
Then
\begin{align}\label{E:basictwistedest}
  \left\|\sqrt{\tau}\dbar u\right\|_{\lambda}^{2}&+2\|\sqrt{\tau}\dbarstar_{\lambda}u\|_{\lambda}^{2}\\
  &\geq
  \int_{\Omega}\left(\tau i\partial\dbar\lambda(u,u)-i\partial\dbar\tau(u,u)
  -\frac{1}{\tau}\left|\left\langle\partial\tau ,u\right\rangle \right|^{2}\right)e^{-\lambda}\;dV\notag
 \end{align}    
  for all $u\in\mathcal{D}^{0,q+1}(\Omega)$.
\end{proposition}

Here, $\|.\|_{\lambda}$ is the norm induced by the inner product $(.,.)_{\lambda}:=\int_{\Omega}\langle.,.\rangle e^{-\lambda}\;dV_{E}$ (on the appropriate form level). Moreover, $\dbarstar_{\lambda}$ is the Hilbert space adjoint of $\dbar$ with respect to $(.,.)_{\lambda}$ so that $\dbarstar_{\lambda}u=e^{\lambda}\dbarstar(e^{-\lambda} u)$ holds for all $u\in\sd(\dbarstar)$.

For brevity, write
   \begin{align*}
        \Theta_{\lambda,\tau}(u,u):=\tau i\partial\dbar\lambda(u,u)-i\partial\dbar\tau(u,u)
        -\frac{1}{\tau}\left|\left\langle\partial\tau,u\right\rangle\right|^{2}.
   \end{align*}

\medskip

\begin{proposition}\label{P:bdduniformest}
  Let $\Omega$ be a bounded, pseudoconvex domain in $\mathbb{C}^{n}$ with smooth boundary, $0\leq q\leq n-1$. 
  Suppose $\Omega$ admits functions $\lambda,\tau\in\mathcal{C}^{2}(\overline{\Omega})$ 
  such that 
  
  \begin{align}\label{E:lowerboundtheta}
    \Theta_{\lambda,\tau}(u,u)\geq c_{1}|u|^{2}e^{-\lambda}\qquad\forall\sjump
    u\in\Lambda^{0,q+1}(\overline{\Omega})
  \end{align}
   for some constant $c_{1}>0$.
  
  Then for each $\alpha\in L^{2}_{0,q+1}(\Omega)$ with  $\dbar_{q+1}\alpha=0$ there is a $v\in\sd(\dbar_{q})$ satisfying
  $\dbar_{q} v=\alpha$ and
  \begin{align*}
     \|v\|_{L^{2}_{0,q}(\Omega)}\leq C\|\alpha\|_{L^{2}_{0,q+1}(\Omega)},
  \end{align*}
  where the constant $C$ equals $\sqrt{2/(c_{1}c_{2})}$ for $c_{2}=\min\{e^{-\lambda(z)}/\tau(z): z\in\overline{\Omega}\}$.
  
  Moreover,
  \begin{align}\label{E:main}
      \|f\|\leq C\|\dbar_{q} f\|\sjump\qquad\forall\sjump f\in\sd(\dbar_{q})\cap\overline{\sr(\dbarstar_{q+1})}.
  \end{align}
\end{proposition}

The proof of Proposition \ref{P:bdduniformest} is analogous to the proof of Theorem 4.1 in \cite{McNeal02}. 
See also the proof of Theorem 4.3 in \cite{Herbig07}.

\medskip

\begin{proof}[Proof of Theorem \ref{P:bdduniformest}]
  Let $\alpha\in L^{2}_{0,q+1}(\Omega)$ with $\dbar\alpha=0$ be given.  Define the linear functional 
     \[
    \begin{array}{ccc}
      F: \Bigl(\bigl\{\sqrt{\tau}\dbarstar_{\lambda}u:u\in\sd(\dbarstar)\bigr\},\, \|.\|_{\lambda}\Bigr)
      & \longrightarrow& \mathbb{C}     \vspace*{0.1cm}\\
      \sqrt{\tau}\dbarstar_{\lambda} u
      &\mapsto&
      (u,\alpha)_{\lambda}
   \end{array}.
    \]
    
  \medskip  
    
 Write $u=u_{1}+u_{2}$ for $u_{1},u_{2}\in L^{2}_{0,q+1}(\Omega)$ with $u_{1}\in\sn(\dbar_{q+1})$ and $u_{2}\perp_{\lambda} 
 \sn(\dbar_{q+1})$ (note that $L^{2}_{0,*}(\Omega)=L^{2}_{0,*}(\Omega,\lambda)$ as 
 $\lambda\in\mathcal{C}^{2}(\overline{\Omega})$). 
 Then for $u\in\sd(\dbarstar_{q+1})$ it follows that
  \begin{align*}
    F\left(\sqrt{\tau}\dbarstar_{\lambda}u\right)=(u_{1},\alpha)_{\lambda}.
  \end{align*}   
  The Cauchy--Schwarz inequality, followed by \eqref{E:lowerboundtheta}, yields
  \begin{align*}
    \left|F(\sqrt{\tau}\dbarstar_{\lambda}u) \right|
    \leq\|u_{1}\|_{2\lambda}\cdot\|\alpha\|
    \leq\frac{1}{\sqrt{c_{1}}}
   \left(\int_{\Omega}\Theta_{\lambda,\tau}(u_{1},u_{1}) e^{-\lambda}\,dV\right)^{1/2}\cdot\|\alpha\|.
 \end{align*}
 Note that $u_{2}\perp_{\lambda}\sn(\dbar_{q+1})$ implies that $u_{2}\in\sd(\dbarstar_{q+1})$ by definition of the latter space 
 (see \eqref{D:domaindbarstar}). 
  Hence 
  $u_{1}\in\sd(\dbarstar_{q+1})$ and \eqref{E:basictwistedest} holds for $u_{1}$. 
  Therefore 
  \begin{align*}  
    \left|F\left(\sqrt{\tau}\dbarstar_{\lambda}u\right) \right|\leq\sqrt{\frac{2}{c_{1}}}\left\|\sqrt{\tau}\dbarstar_{\lambda}u_{1}\right\|_{\lambda}\cdot\|\alpha\|.
  \end{align*}
 The fact that $u_{2}\perp_{\lambda}\sn(\dbar_{q+1})$ entails
  $$\left\|\sqrt{\tau}\dbarstar_{\lambda}u_{1}\right\|_{\lambda}= \left\|\sqrt{\tau}\dbarstar_{\lambda}u\right\|_{\lambda},$$ so that
  \begin{align*}
     \left|F(\sqrt{\tau}\dbarstar_{\lambda}u) \right| \leq\sqrt{\frac{2}{c_{1}}}\left\|\sqrt{\tau}\dbarstar_{\lambda}u\right\|_{\lambda}
     \cdot\|\alpha\|.
  \end{align*}
  That is,  $F$ is a bounded linear functional on 
  $\bigl\{\sqrt{\tau}\dbarstar_{\lambda}u:u\in\sd(\dbarstar_{q+1})\bigr\}$ which is a 
  linear subspace of $L^{2}_{0,q}(\Omega)$.
  The  Hahn--Banach theorem says that $F$ may be extended to a linear functional on $L^{2}_{0,q}(\Omega)$, still denoted $F$,  
  with the same bound. The Riesz representation theorem then yields a unique $w\in L^{2}_{0,q}(\Omega)$ satisfying
  $F(g)=(g,w)_{\lambda}$ for all $g\in L^{2}(\Omega)$ and
   $$ \|w\|_{\lambda}\leq\sqrt{\frac{2}{c_{1}}}\|\alpha\|. $$
  In particular,
  \begin{align*}
    \left(\sqrt{\tau}\dbarstar_{\lambda}u,w\right)_{\lambda}
    =(u,\alpha)_{\lambda}\sjump\qquad\forall\sjump u\in\sd(\dbarstar_{q+1}).
  \end{align*} 
  Since $\sd(\dbarstar_{q+1})$ contains $\Lambda_{c}^{0,q+1}(\Omega)$ which is dense in 
  $L^{2}_{0,q+1}(\Omega)$, $\dbar(\sqrt{\tau}w)=\alpha$ follows. Moreover, setting $v=\sqrt{\tau}w$ yields
 $\dbar_{q} v=\alpha$ with the estimate
 \begin{align*}
   c_{2}\int_{\Omega}|v|^{2}\,dV\leq\int_{\Omega}|v|^{2}\frac{e^{-\lambda}}{\tau}\,dV
   \leq \frac{2}{c_{1}}\int_{\Omega}|\alpha|^{2}\,dV.
 \end{align*}
 
 \medskip
  
  It remains to show that \eqref{E:main} holds. For that, let $f\in\sd(\dbar_{q})\cap
  \overline{\sr(\dbarstar_{q+1})}$ be given. The above derivation yields a $v\in L^{2}_{0,q}(\Omega)$ such that $\dbar_{q} v=\dbar_{q} f$ 
  and
  $\|v\|\leq\sqrt{\frac{2}{c_{1}c_{2}}}\|\dbar_{q} f\|$. 
  Since $$v-f\in\sn(\dbar_{q})=\overline{\sr(\dbarstar_{q+1})}^{\perp},$$ it follows that $f\perp (v-f)$.
  Hence
  \begin{align*}
    \|f\|^{2}\leq\|f\|^{2}+2\text{Re}(v-f,f)+\|v-f\|^{2}=\|v\|^{2}\leq\frac{2}{c_{1}c_{2}}\|\dbar f\|^{2},
  \end{align*}
  which proves \eqref{E:main}.
\end{proof}

There are several ways that the pair of inequalities, \eqref{E:lowerboundtheta} and $c_{2}:=\min\{e^{-\lambda(z)}/\tau(z): z\in\overline{\Omega}\} >0$, can be achieved. We isolate two special cases
that are amenable to application. The first is related to the classical notion of hyperbolicity:

\begin{corollary}\label{C:bounded} 
   If $\Omega$ is a bounded, pseudoconvex domain in $\mathbb{C}^{n}$ with smooth boundary, $0\leq q\leq n-1$, and 
   there exists a $\phi\in \mathcal{C}^2\left(\overline\Omega\right)$ and constants $A, B>0$ such that
   \begin{itemize}
      \item[(i)] $\left|\phi(z)\right| \leq A$ for all $z\in\Omega$,
      \item[(ii)] $i\partial\dbar\phi(u,u)\geq B\, |u|^2$ for all $u\in\Lambda^{0,q+1}\left(\Omega\right)$,
   \end{itemize}
   then $\dbar_q$ has closed range in $L^2_{0,q+1}(\Omega)$.

    Moreover, the constant $C$ in (ii) of Proposition \ref{P:standard} may be taken to be $e^A\sqrt{2/B}$.
\end{corollary}

\begin{proof}
Let $\tau=1$ and $\lambda=\phi$ in Proposition  \ref{P:bdduniformest}.
\end{proof}

For the second special case, we reformulate a definition from \cite{McNeal02}:

\begin{definition}\label{D:selfbded} 
   Let $\Omega$ be a bounded domain in $\mathbb{C}^{n}$ and $0\leq q\leq n-1$. 
   Say that $\psi\in\mathcal{C}^2\left(\Omega\right)$ has a self-bound of $K$ on its complex gradient at level $(0,q)$ if
   \begin{equation}\label{E:self-bounded}
        i\partial\psi\wedge\dbar\psi(u,u)\leq K^2\, i\partial\dbar\psi(u,u)\sjump\quad\forall \sjump u\in\Lambda^{0,q}(\Omega).
    \end{equation}

    Abbreviate \eqref{E:self-bounded} by writing $\left|\partial\psi\right|_{i\partial\dbar\psi}\leq K$ if the form level $(0,q)$ is 
    understood.
\end{definition}

\begin{remark}
If $q=0$, this definition is given in \cite{McNeal02} and $\psi$ is simply said to have self-bounded gradient $\leq K$. The 
abbreviated notation $\left|\partial\psi\right|_{i\partial\dbar\psi}$ in that case corresponds to the natural length measurement of the 
$(1,0)$-form $\partial\psi$
in the Hermitian metric associated to $i\partial\dbar\psi$. Definition \ref{D:selfbded} is given to avoid combinatorial constants 
when $q>0$, in order to focus on the
size of $C$ in Corollary \ref{C:self-bounded} below.
\end{remark}

\begin{corollary}\label{C:self-bounded}
   If $\Omega$ is a bounded, pseudoconvex domain in $\mathbb{C}^{n}$ with smooth boundary, $0\leq q\leq n-1$, and there exists a
   $\psi\in \mathcal{C}^2\left(\overline\Omega\right)$ and constants $D,E>0$ such that
   \begin{itemize}
     \item[(i)] $\left|\partial\psi(z)\right|_{i\partial\dbar\psi} \leq D$ for all $z\in\Omega$,
     \item[(ii)] $i\partial\dbar\psi(u,u)\geq E\, |u|^2$ for all $u\in\Lambda^{0,q+1}\left(\Omega\right)$,
   \end{itemize}
   then $\dbar_q$ has closed range in $L^2_{0,q+1}(\Omega)$.

   Moreover, the constant $C$ in (ii) of Proposition \ref{P:standard} may be taken to be $\frac{\sqrt{2D}}{E}$.
\end{corollary}

\begin{proof} 
  Set $\tau= e^{-\alpha\psi}$ and $\lambda=\alpha\psi$ in Proposition \ref{P:bdduniformest}, for a constant $\alpha >0$ 
  to be determined. Note that for any $\alpha$, $c_{2}:=\min\{e^{-\lambda(z)}/\tau(z): z\in\overline{\Omega}\} =1$.

   For these choices of $\tau, \lambda$, a straightforward computation gives
    \begin{equation}\label{E:easy}
      \Theta_{\lambda, \tau}(u,u) \geq \alpha e^{-\alpha\psi}2 \left(1 -\alpha D^{2}\right) i\partial\dbar\psi(u,u).
   \end{equation}
    Without trying to extract the sharpest lower bound, merely choose $\alpha=1/(2D^{2})$. Then it follows from \eqref{E:easy} that
   $$\Theta_{\lambda, \tau}(u,u) \geq e^{-\alpha\psi}\, \frac {E}{2D^{2}} |u|^2\sjump\qquad\forall\sjump 
   u\in\Lambda^{0,q+1}(\Omega),$$
   and Proposition \ref{P:bdduniformest} completes the proof, with the claimed constant in (ii) of Proposition \ref{P:standard}.
\end{proof}


\section{Approximating Subsets}\label{S:approximatingsd}

The following theorem contains the basic approximation result of pseudoconvex domains:

\begin{theorem}\label{T:approximatingsd}
  Let $\Omega\subset\C^{n}$ be a pseudoconvex domain. Then there exist a 
   sequence 
  $\left\{\Omega_{j}\right\}$ of  open subsets of $\mathbb{C}^{n}$ such that
  \begin{itemize}
      \vspace{0.2cm}
    \item[(i)] $\Omega_{j}\subset\Omega_{j+1}\Subset\Omega$ for all $j\geq 1$,
      \vspace{0.2cm}  
    \item[(ii)] $\bigcup_{j\geq 1}\Omega_{j}=\Omega$,
      \vspace{0.2cm}
     \item[(iii)] for each $j\in\mathbb{N}$ there exists an $m_{j}\in\mathbb{N}$ such that $\Omega_{j}$ is 
     the disjoint union of smoothly bounded, Levi pseudoconvex domains $\Omega_{j}^{k}$, $1\leq k\leq m_{j}$.
   \end{itemize}
\end{theorem}

\begin{proof}
  Since $\Omega$ is pseudoconvex there exists a strictly plurisubharmonic exhaustion function $\Psi\in\mathcal{C}^{\infty}(\Omega)$, 
  see \eqref{D:psc} and the subsequent remark. It follows from the proof of Proposition 2.21 in Chapter II of \cite{range86} that for some 
  $\mathbb{R}$-linear function $\ell(z)$, the function
  \begin{align*}
    r(z):=\Psi(z)+|z|^{2}+\ell(z)
  \end{align*}
  is a smooth, strictly plurisubharmonic exhaustion function whose set of critical points on $\Omega$ is discrete.
  The latter fact, together with $r\in\mathcal{C}^{\infty}(\Omega)$ and the boundedness of  $\{z\in\mathbb{C}^{n}:r(z)<c\}$ for any 
  $c\in\mathbb{R}$, implies that for any $j\in\mathbb{N}$ there is a $j^{*}\in(j-1/2,j+1/2)$ such that $\nabla r(z)\neq 0$ 
  whenever $r(z)=j^{*}$. Set
  \begin{align*}
    \Omega_{j}=\left\{z\in\mathbb{C}^{n}:r(z)<j^{*}\right\}.
  \end{align*}
  Both (i) and (ii) then follow straightforwardly from the fact that $r$ is an exhaustion function.
 
  Since each critical point of $r$ is isolated, it follows that any convergent sequence $\{x_{n}\}$ of critical points satisfies 
  $\lim_{n\to\infty}|x_{n}|=\infty$. Therefore, any bounded subset of $\Omega$ contains  finitely many critical points of $r$.
  This implies that $\Omega_{j}$ has finitely many connected components: $\Omega_{j}^{k}$ for $1\leq k\leq m_{j}$ for some 
  $m_{j}\in\mathbb{N}$. The fact that $\nabla r\neq 0$ on $b\Omega_{j}$ implies that the intersection of the closures of any two 
  components of $\Omega_{j}$ are mutually disjoint. It also implies that $b\Omega_{j}^{k}$ has smooth boundary for $1\leq k\leq m_{j}$. 
  That all $\Omega_{j}^{k}$ are Levi pseudoconvex now follows from $r$ being a strictly plurisubharmonic function in $\Omega$.
\end{proof}
The sets $\left\{\Omega_j\right\}$ described in Theorem \ref{T:approximatingsd}  will be called a sequence of approximating 
subsets for $\Omega$.

\medskip


\section{$\dbar_{q}$ on general pseudoconvex domains}\label{S:unbounded}

Proposition \ref{P:bdduniformest} is stated for bounded pseudoconvex domains with smooth boundary and functions $\lambda, \tau\in\mathcal{C}^{2}\left(\overline{\Omega}\right)$ because its
proof hinges on Proposition \ref{P:twistedestimate}; this result requires these hypotheses. Extending \eqref{E:basictwistedest} to unbounded or non-smooth domains, and
to $\lambda, \tau$ not necessarily smooth up to $b\Omega$, requires dealing with density issues between $\mathcal{D}^{0,q+1}(\Omega)$ and $\sd\left(\dbarstar\right)\cap\sd\left(\dbar\right)$.
An extension of this kind would be delicate and would not be universally valid (it would depend on the exact lack of smoothness or unboundedness of the data).

The closed range inequality, \eqref{E:standard}, is less delicate. We show the conclusion of Proposition \ref{P:bdduniformest} holds on a general pseudoconvex domain $\Omega$, if
$\Omega$ admits
functions $\lambda, \tau\in\mathcal{C}^{2}\left(\Omega\right)$ satisfying the hypotheses of Proposition \ref{P:bdduniformest} uniformly. The additional ingredient is the following

\begin{lemma}\label{L:closedrangeapprox}
  Let $\Omega\subset\mathbb{C}^{n}$ be a pseudoconvex domain. Let $\{\Omega_{j}\}_{j}$ be a sequence 
  of approximating subsets for $\Omega$. Suppose there exists a constant $C>0$ such that for all 
  $\alpha_{j}\in L_{0,q+1}^{2}(\Omega_{j})\cap\sn(\dbar_{q+1})$ there exists a 
  $v_{j}\in L_{0,q}^{2}(\Omega_{j})\cap(\sn(\dbar_{q}))^{\perp}$ with $\dbar_{q}v_{j}=\alpha_{j}$ on $\Omega_{j}$ 
  in the distributional sense and 
  \begin{align}\label{E:uniform}
    \|v_{j}\|_{L^{2}_{0,q}(\Omega_{j})}\leq C\|\alpha_{j}\|_{L^{2}_{0,q+1}(\Omega_{j})}.
  \end{align}
  Then $\dbar_{q}$ has closed range on $\Omega$.
\end{lemma}

Recall that $\Omega_{j}=\bigcup_{\ell=1}^{m_{j}}\Omega_{j}^{\ell}$. So $\alpha_{j}\in L^{2}(\Omega_{j})$ is to mean that 
$\alpha_{j}=\alpha_{j}^{1}+\dots+\alpha_{j}^{m_{j}}$ for $\alpha_{j}^{\ell}\in L^{2}(\Omega_{j}^{\ell})$
and $\|\alpha_{j}\|_{L^{2}_{0,q}(\Omega_{j})}^{2}=\sum_{\ell=1}^{m_{j}}\|\alpha_{j}^{\ell}\|_{L^{2}_{0,q}(\Omega_{j}^{\ell})}^{2}$. Similarly,
$\dbar_{q}$ on $L^{2}_{0,q}(\Omega_{j})$ is the direct sum of the $\dbar_{q}$-operators associated to the $\Omega_{j}^{\ell}$'s.

\medskip

\begin{proof}
  It will be shown that 
  \begin{align}
    \|u\|_{L^{2}_{0,q}(\Omega)}\leq C\|\dbar_{q}u\|_{L^{2}_{0,q+1}(\Omega)}
    \sjump\qquad\forall\sjump u\in\sd(\dbar_{q})\cap(\sn(\dbar_{q}))^{\perp}.
  \end{align}
  
  Suppose $u\in\sd(\dbar_{q})\cap(\sn(\dbar_{q}))^{\perp}$. 
  By the Definition of $\sd(\dbar_{q})$, there exists a sequence 
  $\{u_{j}\}_{j}\subset\Lambda^{0,q}(\overline{\Omega})$ such that $u_{j}\longrightarrow u$ in 
  $L^{2}_{0,q}(\Omega)$ and $\dbar u_{j}\longrightarrow\dbar u$ in $L^{2}_{0,q+1}(\Omega)$ as $j$ approaches $\infty$. 
  By hypothesis of the Lemma, there exists a sequence $\{v_{j}\}_{j}\subset L^{2}_{0,q}(\Omega_{j})\cap(\sn(\dbar_{q}))^{\perp}$ such that
  $\dbar_{q} v_{j}=(\dbar_{q}u_{j})$ on $\Omega_{j}$ and a constant $C>0$ such that
  \begin{align}\label{E:vjujestimate}
    \|v_{j}\|_{L^{2}_{0,q}(\Omega_{j})}\leq C\|\dbar u_{j}\|_{L^{2}_{0,q+1}(\Omega_{j})}.
  \end{align}  
  
  The estimation of $\|u\|_{L^{2}_{0,q}(\Omega)}$ proceeds as follows:
  \begin{align}\label{E:uestimate}
      \|u\|_{L^{2}_{0,q}(\Omega)}&
      \leq \|u\|_{L^{2}_{0,q}(\Omega_{j})}+\|u\|_{L^{2}_{0,q}(\Omega\setminus\Omega_{j})}\notag\\
      &\leq\|u-u_{j}\|_{L^{2}_{0,q}(\Omega_{j})}+\|u_{j}\|_{L^{2}_{0,q}(\Omega_{j})}
      +\|u\|_{L^{2}_{0,q}(\Omega\setminus\Omega_{j})}\notag\\
      &\leq\|u-u_{j}\|_{L^{2}_{0,q}(\Omega_{j})}+\|u_{j}-v_{j}\|_{L^{2}_{0,q}(\Omega_{j})}+\|v_{j}\|_{L^{2}_{0,q}(\Omega_{j})}
      +\|u\|_{L^{2}_{0,q}(\Omega\setminus\Omega_{j})}.
  \end{align}
  Inequality \eqref{E:vjujestimate} gives
  \begin{align*}
    \|v_{j}\|_{L^{2}_{0,q}(\Omega_{j})}\leq C\|\dbar u_{j}\|_{L^{2}_{0,q+1}(\Omega_{j})}
    \leq C\left(\|\dbar u\|_{L^{2}_{0,q+1}(\Omega)}+\|\dbar u-\dbar u_{j}\|_{L^{2}_{0,q+1}(\Omega_{j})} \right).
  \end{align*}
  Hence, for any $\epsilon>0$ there is a $j^{*}$ such that for all $j\geq j^{*}$ \eqref{E:uestimate} becomes
  \begin{align*}
    \|u\|_{L^{2}_{0,q}(\Omega)}\leq \epsilon+C\|\dbar u\|_{L^{2}_{0,q+1}(\Omega)}+
    \|u_{j}-v_{j}\|_{L_{0,q}^{2}(\Omega_{j})}.
  \end{align*}
  
  To conclude the proof, it suffices to show that $\{u_{j}-v_{j}\}_{j}$ has a subsequence $\{u_{j_{k}}-v_{j_{k}}\}_{j_{k}}$ 
  such that $\|u_{j_{k}}-v_{j_{k}}\|_{L^{2}_{0,q}(\Omega_{j_{k}})}$ converges to $0$ as $j_{k}\to\infty$. 
  Notice first that
  $v_{j}\perp\sn(\dbar_{q})$ on $L^{2}_{0,q}(\Omega_{j})$ while $u_{j}-v_{j}\in\sn(\dbar_{q})$ on $L^{2}_{0,q}(\Omega_{j})$. 
  Therefore,
  \begin{align}\label{E:uestimate2}
    \|u_{j}-v_{j}\|_{L^{2}_{0,q}(\Omega_{j})}^{2}&=\left(u_{j}-v_{j},u_{j} \right)_{L^{2}_{0,q}(\Omega_{j})}\\
    &\leq\left|\left(u_{j}-v_{j},u-u_{j} \right)_{L^{2}_{0,q}(\Omega_{j})} \right|+\left|\left(u_{j}-v_{j},u \right)_{L^{2}_{0,q}(\Omega_{j})}
    \right|\notag\\
    &\leq\|u_{j}-v_{j}\|_{L^{2}_{0,q}(\Omega_{j})}\cdot\left(
    \|u-u_{j}\|_{L^{2}_{0,q}(\Omega_{j})}+\|u\|_{L^{2}_{0,q}(\Omega_{j})}
    \right).\notag
  \end{align}
  Since $u\in L^{2}_{0,q}(\Omega)$ and $u_{j}\longrightarrow u$ in $L^{2}_{0,q}(\Omega)$, it follows that 
  $\|u_{j}-v_{j}\|_{L^{2}_{0,q}(\Omega_{j})}$ is uniformly bounded in $j$. Let $\widetilde{u_{j}}=u_{j}$ and 
  $\widetilde{v_{j}}=v_{j}$ on 
  $\Omega_{j}$ and $\widetilde{u_{j}}=0=\widetilde{v_{j}}$ on $\Omega_{j}^{c}$. 
  It then follows that 
  $\{\widetilde{u_{j}}-\widetilde{v_{j}}\}_{j}$ is a bounded sequence in $L_{0,q}^{2}(\Omega)$. Hence, it has a subsequence 
  $\{\widetilde{u_{j_{k}}}-\widetilde{v_{j_{k}}}\}_{j_{k}}$ which is weakly convergent to some $w\in L^{2}_{0,q}(\Omega)$.
  That is, for all $g\in L^{2}_{0,q}(\Omega)$
  \begin{align*}
    \left(\widetilde{u_{j_{k}}}-\widetilde{v_{j_{k}}}-w,g \right)_{L^{2}_{0,q}(\Omega)}\longrightarrow 0\text{\;\;as\;\;}j_{k}\to\infty.
  \end{align*}
  Since
  $\left(\widetilde{u_{j_{k}}}-\widetilde{v_{j_{k}}},g \right)_{L^{2}_{0,q}(\Omega)}
  =(u_{j_{k}}-v_{j_{k}},g)_{L^{2}_{0,q}(\Omega_{j_{k}})}$ for all $g\in L^{2}_{0,q}(\Omega)$,
  it follows  that 
  \begin{align}\label{E:weakconvergence}
    (u_{j_{k}}-v_{j_{k}}-w,g)_{L^{2}_{0,q}(\Omega_{j_{k}})}\longrightarrow 0\text{\;\;as\;\;}j_{k}\to\infty.
   \end{align} 
   
   \medskip
   
  Repeating the 
  arguments in \eqref{E:uestimate2}, with $u_{j_{k}}-v_{j_{k}}$ in place of $u_{j}-v_{j}$, we obtain
  \begin{align}\label{E:uestimate3}
      \|u_{j_{k}}-v_{j_{k}}\|_{L^{2}_{0,q}(\Omega_{j_{k}})}^{2}
      &\leq \left|\left(u_{j_{k}}-v_{j_{k}},u-u_{j_{k}}  \right)_{L_{0,q}^{2}(\Omega_{j_{k}})}\right|
      +\left|\left(u_{j_{k}}-v_{j_{k}},u \right)_{L_{0,q}^{2}(\Omega_{j_{k}})}  \right|\\
      &\leq
      \left\|u_{j_{k}}-v_{j_{k}}\right\|_{L_{0,q}^{2}(\Omega_{j_{k}})}\cdot
      \left\|u-u_{j_{k}} \right\|_{L^{2}_{0,q}(\Omega)}+\left|\left(u_{j_{k}}-v_{j_{k}},u \right)_{L_{0,q}^{2}(\Omega_{j_{k}})}  \right|.\notag
   \end{align}  
   Thus, in order to prove that $u_{j_{k}}-v_{j_{k}}$ goes to $0$ in $L^{2}_{0,q}(\Omega_{j_{k}})$
   it 
   suffices to show 
   $$\left(u_{j_{k}}-v_{j_{k}},u \right)_{L_{0,q}^{2}(\Omega_{j_{k}})} \longrightarrow 0\text{\;\;as\;\;}j_{k}\to\infty,$$
   since $\{u_{j_{k}}-v_{j_{k}}\}$ is bounded in $L^{2}(\Omega_{j_{k}})$ and $u_{j_{k}}\longrightarrow u$ in $L^{2}_{0,q}(\Omega)$.
   But
   \begin{align*}
     \left|\left(u_{j_{k}}-v_{j_{k}},u \right)_{L_{0,q}^{2}(\Omega_{j_{k}})}  \right|
     \leq
     \left|\left(u_{j_{k}}-v_{j_{k}}-w,u  \right)_{L_{0,q}^{2}(\Omega_{j_{k}})}\right|+
     \left|\left(w,u  \right)_{L_{0,q}^{2}(\Omega_{j_{k}})}\right|.
   \end{align*}
  The first term tends to 0 as $j_{k}\to\infty$ by \eqref{E:weakconvergence}, so it suffices to show that $\left(w,u  \right)_{L_{0,q}^{2}(\Omega_{j_{k}})}\longrightarrow 0$ as $j_{k}\to\infty$. The Montone Convergence theorem reduces this showing that $\left(w,u  \right)_{L_{0,q}^{2}(\Omega)}=0$, since $\Omega$ is the union of the increasing subsets $\{\Omega_{j_{k}}\}$. However, $u$ is assumed to be in $\sn(\dbar_{q})^{\perp}$ and it is almost obvious that $w\in\sn(\dbar_{q})$ (being the weak limit of elements in
  $\sn(\dbar_{q})$).  The details that $w\in\sn(\dbar_{q})$ are given for convenience. Denote by $\vartheta_{q+1}$ the formal adjoint of $\dbar_{q}$ and let $\varphi\in\Lambda_{c}^{0,q+1}(\Omega)$. Choose $j_{0}\in\mathbb{N}$ sufficiently large such that
  the compact support of $\varphi$ is contained in $\Omega_{j_{k}}$ for all $j_{k}\geq j_{0}$. Then 
  \begin{align}\label{E:winkernel}
    \int_{\Omega}\left\langle w,\vartheta_{q+1}\varphi\right\rangle\,dV
    &=\int_{\Omega_{j_{k}}}\left\langle w,\vartheta_{q+1}\varphi\right\rangle\,dV\notag\\
    &=\int_{\Omega_{j_{k}}}\left\langle w-(u_{j_{k}}-v_{j_{k}}),\vartheta_{q+1}\varphi\right\rangle\,dV
    +\int_{\Omega_{j_{k}}}\left\langle u_{j_{k}}-v_{j_{k}},\vartheta_{q+1}\varphi\right\rangle\,dV.
  \end{align}
  The last term in \eqref{E:winkernel} equals zero since $u_{j_{k}}-v_{j_{k}}\in\sn(\dbar_{q})$ while the first term tends to $0$ as $j_{k}\to\infty$ because \eqref{E:weakconvergence} holds. Thus, $w\in\sd(\dbar_{q})$ and $\dbar_{q}w=0$ (see the definition of $\sd(\dbar_{q})$ above \eqref{D:domaindbar2}). This concludes the proof of $u_{j_{k}}-v_{j_{k}}$ tending to 0 in $L^{2}_{0,q}(\Omega_{j_{k}})$ as $j_{k}\to\infty$.
  
   Repeating the arguments, starting at \eqref{E:uestimate}, with $u_{j_{k}}$ and $v_{j_{k}}$ 
  in place of $u_{j}$ and $v_{j}$, respectively, completes the proof of the Lemma.
  \end{proof}
  
  Our main result, essentially Proposition \ref{P:bdduniformest} under relaxed hypotheses, follows easily:
  
  \begin{theorem}\label{T:bdduniformest}
  Let $\Omega\subset\mathbb{C}^{n}$ be a pseudoconvex domain and $0\leq q\leq n-1$. 
  Suppose $\Omega$ admits functions $\lambda,\tau\in\mathcal{C}^{2}(\Omega)$ and constants $c_1, c_2 >0$
  such that 
  
  \begin{itemize}
 \item[(i)]  $\Theta_{\lambda,\tau}(u,u)\geq c_{1}|u|^{2}e^{-\lambda}\qquad\forall\sjump
    u\in\Lambda^{0,q+1}(\Omega)$
 \medskip
  \item[(ii)] $\min\left\{e^{-\lambda(z)}/\tau(z): z\in\Omega\right\}\geq c_2$,
  \end{itemize}
   then $\dbar_q$ has closed range in $L^2_{0,q+1}(\Omega)$.
  
\end{theorem}
  
\begin{proof} Let $\left\{\Omega_j\right\}$ be the sequence of approximating subsets for $\Omega$ given by Theorem 
\ref{T:approximatingsd}. Note that $\lambda, \tau\in\mathcal{C}^{2}\left(\overline\Omega_j\right)$. Proposition  \ref{P:bdduniformest} applies,
giving a uniform $C$ such that \eqref{E:uniform} holds. Lemma \ref{L:closedrangeapprox} completes the proof.
\end{proof}

Generalized versions of Corollaries \ref{C:bounded} and \ref{C:self-bounded} follow as before:

\begin{corollary}\label{C:bounded1} 
   If $\Omega\subset\C^n$ is a pseudoconvex domain, $0\leq q\leq n-1$, and 
   there exists a $\phi\in \mathcal{C}^2\left(\Omega\right)$ and constants $A, B>0$ such that
   \begin{itemize}
      \item[(i)] $\left|\phi(z)\right| \leq A$ for all $z\in\Omega$,
      \item[(ii)] $i\partial\dbar\phi(u,u)\geq B\, |u|^2$ for all $u\in\Lambda^{0,q+1}\left(\Omega\right)$,
   \end{itemize}
   then $\dbar_q$ has closed range in $L^2_{0,q+1}(\Omega)$.

\end{corollary}

\begin{corollary}\label{C:self-bounded1}
   If $\Omega\subset\C^n$ is a pseudoconvex domain, $0\leq q\leq n-1$, and there exists a 
   $\psi\in \mathcal{C}^2\left(\Omega\right)$ and constants $D,E>0$ such that
   \begin{itemize}
     \item[(i)] $\left|\partial\psi(z)\right|_{i\partial\dbar\psi} \leq D$ for all $z\in\Omega$,
     \item[(ii)] $i\partial\dbar\psi(u,u)\geq E\, |u|^2$ for all $u\in\Lambda^{0,q+1}\left(\Omega\right)$,
   \end{itemize}
   then $\dbar_q$ has closed range in $L^2_{0,q+1}(\Omega)$.

\end{corollary}

\medskip

\section{Examples}\label{S:examples}

\subsection{Dimension 1}\label{SS:dim1}

If $z\in\C$ and $L>0$, let $\mathbb{D}(z,L)$ denote the disc centered at $z$ of radius $L$.

\begin{definition}\label{D:largedisks}
 (i) A domain $\Omega\subset\mathbb{C}$ is said to not contain arbitrarily large discs if there exists an $L>0$ such 
    that 
    \begin{align}\label{E:nolargedisks}
       \mathbb{D}(z,L)\cap\overline{\Omega}^{c}\neq\emptyset\sjump\qquad\forall\sjump z\in\Omega.
    \end{align}
 (ii) A domain $\Omega\subset\mathbb{C}$ is said to satisfy condition $\sx$ if there exist an $M>0$ and a $  
  \delta>0$ such that for all $z\in\Omega$ there exists a $z^{*}\in\mathbb{D}(z,M)\cap\overline{\Omega}^{c}$ such that the 
  distance of $z^{*}$ to $\overline{\Omega}$ is greater than $\delta$.
\end{definition}
Note that condition $\sx$  in (ii) above implies that \eqref{E:nolargedisks} holds with $M$ in place of $L$. Hence, condition $\sx$ is only satisfied 
by domains which do not contain arbitrarily large discs.

\begin{lemma}\label{L:noarblargedisks}
  Let $\Omega\subset\mathbb{C}$ be a domain. If $\dbar_{0}$ has closed range on $L^{2}_{0,1}(\Omega)$, then 
  $\Omega$ cannot contain arbitrarily large discs.
\end{lemma}

\begin{proof} 
  Suppose $\Omega$ contains arbitrarily large discs. Then there exists a sequence $\{z_{j}\}_{j}\subset\Omega$ such that  
  $\mathbb{D}(z_{j},j)\Subset\Omega$ for all $j\in\mathbb{N}$.
  Let $\alpha\in\mathcal{C}^{\infty}_{c}(\mathbb{D}(0,1))$ be not identically zero. Set $\alpha_{j}=\alpha((z-z_{j})/j)$, so that 
  $\alpha_{j}\in\mathcal{C}^{\infty}_{c}(\mathbb{D}(z_{j},j))$. Hence $\alpha_{j}\in\sd(\dbarstar)$. Set $u_{j}=\dbarstar\alpha_{j}$. 
  Then $u_{j}\in\sd(\dbar)$ and
  \begin{align*}
    \|u_{j}\|_{L^{2}(\Omega)}^{2}=\left\|\dbarstar\alpha_{j}\right\|_{L^{2}(\mathbb{D}(z_{j},j))}^{2}&=
    \frac{1}{j^{2}}\int_{\mathbb{D}(z_{j},j)}\left|\left(\dbarstar\alpha\right)\left(\frac{z-z_{j}}{j}\right) \right|^{2}\;dV(z)\\
    &=\int_{\mathbb{D}(0,1)}\left|\dbarstar\alpha(z) \right|^{2}\;dV(z)=\left\|\dbarstar\alpha \right\|_{L^{2}(\mathbb{D}(0,1))}^{2}.
  \end{align*}
  Furthermore,
  \begin{align*}
    \left\|\dbar u_{j}\right\|_{L^{2}(\Omega)}^{2}
    = \frac{1}{j^{4}}\int_{\mathbb{D}(z_{j},j)}\left|\left(\dbar\dbarstar\alpha\right)\left(\frac{z-z_{j}}{j}  \right)\right|^{2}\;dV(z)
    =\frac{1}{j^{2}}\left\|\dbar\dbarstar\alpha \right\|_{L^{2}(\mathbb{D}(0,1))}.
  \end{align*}
  Since $0< k_1 <\left\|\dbarstar\alpha\right\|, \left\|\dbar\dbarstar\alpha\right\| <k_2$ for constants independent of $j$, it follows that there is no constant $C>0$ such that 
  \begin{align*}
    \left\|u_{j}\right\|_{L^{2}(\Omega)}\leq C\left\|\dbar u_{j}\right\|_{L^{2}(\Omega)}
  \end{align*}
  holds. Thus $\dbar$ does not have closed range on $L^{2}_{0,1}(\Omega)$.
\end{proof}

\medskip

\begin{proposition}\label{P:closedrangesuff}
   Let $\Omega\subset\mathbb{C}$ be a domain. If $\Omega$ satisfies condition $\sx$, then $\dbar_{0}$ has closed range on $L^{2}_{0,1}(\Omega)$.
\end{proposition}
  We prove Proposition \ref{P:closedrangesuff} by constructing a function $\phi$ in Lemma \ref{L:philatticeconstruction} satisfying the hypotheses of Corollary \ref{C:bounded1}.
  For this construction, a perturbation of a lattice in a neighborhood of the closure 
  of $\Omega$ is used.
  
\begin{lemma}\label{L:grid1dim}
  Let $\Omega\subset\mathbb{C}$ be a domain satisfying condition $\sx$. Then 
  there exist constants $M,\delta>0$ and $\Lambda\subset (M\mathbb{Z})^{2}$ such that
 \begin{itemize}
       \item[(a)] $\mathbb{D}(w,M)\cap\overline{\Omega}^{c}\neq\emptyset\neq \mathbb{D}(w,M)\cap\Omega$ for all  
         $w\in\Lambda$,
         \vspace{0.3cm}
       \item[(b)] $\Omega\subset\bigcup_{w\in\Lambda}\mathbb{D}(w, M)$,
         \vspace{0.3cm}
      \item[(c)] for each $w\in\Lambda$ there exists a $w^{*}\in\overline{\Omega}^{c}$ satisfying
         \vspace{0.2cm}
      \begin{itemize}
        \item[(i)] the distance of $w^{*}$ to $\overline{\Omega}$ is greater than $\delta$,
          \vspace{0.2cm}
        \item[(ii)] $\mathbb{D}(w, M)\subset\mathbb{D}(w^{*}, 3M)$.
      \end{itemize}
    \end{itemize}
\end{lemma}

\begin{proof}[Proof of Lemma \ref{L:grid1dim}]
  Let $M,\delta>0$ be as in part (ii) of Definition \ref{D:largedisks}. Set $\Lambda=\{w\in(M\mathbb{Z})^{2}: 
  \mathbb{D}(w,M)\cap\Omega\neq\emptyset\}$. Then (a) follows.
  
  For each $z\in\Omega$ there is a $w_{0}\in(M\mathbb{Z})^{2}$ contained in $\mathbb{D}(z,M)$. Hence $z\in\mathbb{D}
  (w_{0},M)$. This means in particular that $\mathbb{D}(w_{0},M)\cap\Omega\neq\emptyset$. Thus $w_{0}\in\Lambda$ and (b) 
  follows.
  
  Let $w\in\Lambda$. Assume first that $w\in\Omega$. Since $\Omega$ satisfies condition $\sx$, there exists 
  a $w^{*}\in\mathbb{D}(w,M)$ such that the distance of $w^{*}$ to $\overline{\Omega}$ is greater than $\delta$. 
  Moreover, $\mathbb{D}(w,M)\subset\mathbb{D}(w^{*},2M)$. That is, both (i) and (ii) hold in this case. Now assume that $w\in
  \overline{\Omega}^{c}$. Then, since $w\in\Lambda$, there exists a $z\in\Omega$ contained in $\mathbb{D}(w,M)$. By 
  assumption there is a $w^{*}\in\mathbb{D}(z,M)\cap\overline{\Omega}^{c}$ whose distance to $\overline{\Omega}
  $ is greater than $\delta$. If $y\in\mathbb{D}(w,M)$ then
  $$|y-w^{*}|\leq|y-w|+|w-z|+|z-w^{*}|\leq 3M,$$
  i.e., (ii) holds in this case. This concludes the proof.
\end{proof}

\begin{lemma}\label{L:philatticeconstruction}
  Let $\Omega\subset\mathbb{C}$ be a domain which satisfies condition $\sx$. Then 
 there exists a $\phi\in \mathcal{C}^2\left(\Omega\right)$ and constants $A, B>0$ such that
   \begin{itemize}
      \item[(i)] $\left|\phi(z)\right| \leq A$ for all $z\in\Omega$,
      \item[(ii)] $i\partial\dbar\phi(u,u)\geq B\, |u|^2$ for all $u\in\Lambda^{0,q+1}\left(\Omega\right)$.
   \end{itemize}
\end{lemma}

\begin{proof}
 Let $M$, $\delta>0$ be as in (ii) of Definition \ref{D:largedisks} and $\Lambda\subset(M\mathbb{Z})^{2}$ as in Lemma
 \ref{L:grid1dim}. Write $w_{\ell,k}=\ell M+ikM$ for $\ell$, $k\in\mathbb{Z}$. For each $w_{\ell,k}\in\Lambda$ choose a 
 $w_{\ell,k}^{*}\in\overline{\Omega}^{c}$ as described in (c) of Lemma \ref{L:grid1dim}, set $\Lambda^{*}=\{w_{\ell,k}^{*}:w_{\ell,k}
 \in\Lambda\}$. It will be shown that
 \begin{align}\label{E:seriesphi}
    \sum_{w_{\ell,k}^{*}\in\Lambda^{*}}\left|z-w_{\ell,k}^{*} \right|^{-4}
 \end{align}
 converges to a function, $\phi(z)$, of class $\mathcal{C}^{\infty}$ on $\overline{\Omega}$ satisfying conditions (i) and (ii).

Given $\gamma\in\mathbb{N}$ and $p=p_{1}+ip_{2}\in\mathbb{C}$ denote by $\mathcal{Q}_{\gamma M}^{p}$ the set 
$$\left\{q_{1}+iq_{2}\in\mathbb{C}: |q_{j}-p_{j}|\leq M, j\in\{1,2\}\right\}$$ and $\mathcal{Q}_{0 M}^{p}=\{p\}$.

Fix $z\in\overline{\Omega}$. Then there is a $w_{\ell_{0},k_{0}}\in\Lambda$ such that $z\in\mathbb{D}(w_{\ell_{0},k_{0}},M)$. To show convergence of the series \eqref{E:seriesphi} at $z$, let us sum over the sets
\begin{align*}
  \mathcal{A}_{0 M}^{w_{\ell_{0},k_{0}}}:=\Lambda^{*}\cap
  \mathcal{Q}_{0M}^{w_{\ell_{0},k_{0}}},\qquad\text{and}\sjump
  \mathcal{A}_{\gamma M}^{w_{\ell_{0},k_{0}}}:=\Lambda^{*}\cap
  \left(\mathcal{Q}_{\gamma M}^{w_{\ell_{0},k_{0}}}
  \setminus\mathcal{Q}_{(\gamma-1) M}^{w_{\ell_{0},k_{0}}}
  \right)\sjump\text{for}\sjump\gamma\in\mathbb{N}.
\end{align*}
To compute the cardinality of $\mathcal{A}_{\gamma M}^{w_{\ell_{0},k_{0}}}$, note first that
\begin{align*}
  \left|\Lambda\cap\mathcal{Q}_{\gamma M}^{w_{\ell,k}} \right|
  \leq\left(2\gamma+1\right)^{2}.
\end{align*}
Moreover, it follows from (ii) of part (c) of Lemma \ref{L:grid1dim} that
\begin{align*}
  \left|\Lambda^{*}\cap\mathcal{Q}_{\gamma M}^{w_{\ell_{0},k_{0}}} \right|\leq
  \left|\Lambda^{*}\cap\mathcal{Q}_{(\gamma+3)M}^{w_{\ell_{0},k_{0}}} \right|\leq(2\gamma +7)^{2}\sjump\qquad\forall\sjump\gamma\in\mathbb{N},
\end{align*}
and
\begin{align*}
   \left|\Lambda^{*}\cap\mathcal{Q}_{(\gamma-1)M}^{w_{\ell_{0},k_{0}}} \right|\geq
  \left|\Lambda^{*}\cap\mathcal{Q}_{(\gamma-4)M}^{w_{\ell_{0},k_{0}}} \right|\geq(2\gamma -7)^{2}\sjump\qquad\forall\sjump
  \gamma\geq 4.
\end{align*}
Therefore
\begin{align}\label{E:annuluscount}
  \left| \mathcal{A}_{\gamma M}^{w_{\ell_{0},k_{0}}}\right|&\leq\left\{
  \begin{array}{l l}
    (2\gamma+7)^{2} &\;\text{for}\;0\leq\gamma\leq 3\\
    (2\gamma+7)^{2}-(2\gamma-7)^{2} &\;\text{for}\;\gamma\geq 4
  \end{array}
  \right. .
  \end{align}
  Hence 
  \begin{align*}
    0\leq \sum_{\gamma=4}^{\infty}\;\sum_{w_{\ell,k}^{*}\in\mathcal{A}_{\gamma M}^{w_{\ell_{0},k_{0}}}}
    \left|z-w_{\ell,k}^{*}\right|^{-4}\leq \frac{56}{M^{4}}\sum_{\gamma\geq 4}^{\infty}\frac{1}{\gamma^{3}},
  \end{align*} 
  which implies that the series  \eqref{E:seriesphi} converges absolutely to some scalar, $\phi(z)$, at $z$. In fact, the 
  convergence of the series to $\phi$ is uniform on $\overline{\Omega}$. Therefore $\phi$ is continuous on $\overline{\Omega}$.  
  Similarly, since any $k$-th derivative of $|z-w_{\ell,k}^{*}|^{-4}$ is of order $\mathcal{O}(|z-w_{\ell,k}^{*}|^{-(4+k)})$, it follows 
  that $\phi\in\mathcal{C}^{\infty}(\overline{\Omega})$ and derivatives of $\phi$ may be computed term by term. The latter implies 
  that
  \begin{align*}
    \phi_{z\bar{z}}(z)\geq 4\left|z-w_{\ell,k}^{*} \right|^{-6}\geq\frac{4}{(3M)^{6}}=:B,
  \end{align*}
  where $w_{\ell,k}$ is such that $z\in\mathbb{D}(w_{\ell,k},M)\cap\overline{\Omega}$ (see (c) of Lemma \ref{L:grid1dim}). Thus 
  (ii) is shown to hold for $\phi$.
  That $\phi$ is  bounded on $\overline{\Omega}$ also follows from  \eqref{E:annuluscount}:
  \begin{align*}  
    \phi(z)\leq 49\delta^{-4}+\frac{1}{M^{4}}\left(\sum_{\gamma=1}^{3}\frac{(2\gamma+7)^{2}}{\gamma^{4}}
    +56\sum_{\gamma=4}^{\infty}\frac{1}{\gamma^{3}}\right)=:A,
\end{align*}  
 i.e., (i) holds for $\phi$ as well.  
  \end{proof}

  \medskip
  
  \subsubsection{A class of examples} 
  
  Let $\{c_{j}\}_{j\in\mathbb{Z}}$ be a strictly increasing sequence of real numbers such that 
  $\lim_{|j|\to\infty}|c_{j}|=\infty$ and $\sup_{j\in\mathbb{Z}}(c_{j}-c_{j-1})\leq M$ for some $M>0$. Let $\eta_{j}\in\mathcal{C}^{\infty}(\mathbb{R})$ for $j\in\mathbb{Z}$ such that
  \begin{itemize}
    \item[(a)]  $\eta_{2j}(x)<c_{j}<\eta_{2j+1}(x)$ for all $x\in\mathbb{R}$, $j\in\mathbb{Z}$,\vspace{0.3cm}
    \item[(b)] $\eta_{2j-1}(x)<\eta_{2j}(x)$ for all $x\in\mathbb{R}$, $j\in\mathbb{Z}$.
  \end{itemize}
  Define $\mathcal{S}_{j}=\left\{z\in\mathbb{C}:\eta_{2j-1}\left(\re(z)\right)<\im(z)<\eta_{2j}\left(\re(z)\right) \right\}$ 
  for $j\in\mathbb{Z}$, set $\mathcal{S}=\bigcup_{j\in\mathbb{Z}}\mathcal{S}_{j}$. It is straightforward to check that $\mathcal{S}$ satisfies the following properties:
  \begin{itemize} 
    \item[(i)] $\mathcal{S}$ is an open set with smooth boundary --  however, $\mathcal{S}$ is not connected.
    \item[(ii)] $\mathcal{S}$ does not contain arbitrarily large discs.
    \item[(iii)] $\mathcal{S}$ satisfies condition $\sx$ if and only if there is a strictly increasing 
    subsequence 
    $\{j_{k}\}_{k\in\mathbb{Z}}$ in $\mathbb{Z}$ such that $\lim_{|k|\to\infty}|j_{k}|=\infty$, $\sup(c_{j_{k+1}}-c_{j_{k}})\leq L$ for 
    some 
    $L>0$, and the distance between $\mathcal{S}_{j_{k}}$ and $\mathcal{S}_{j_{k}+1}$ is uniformly bounded from below by some positive constant $\delta>0$.
  \end{itemize}
  
  \begin{lemma}\label{L:stripweight}
    There exists a $\varphi\in\mathcal{C}^{2}(\mathcal{S})$ satisfying conditions (i) and (ii) of 
    Corollary~\ref{C:bounded1}.
  \end{lemma}
 \begin{proof}
  For each $j\in\mathbb{Z}$ choose a real-valued, smooth function 
  $\varphi_{j}$ with compact support in $\{z\in\mathbb{C}:c_{j-1}<\im(z)<c_{j}\}$ such that $\varphi_{j}(z)=(\im(z)-c_{j})^{2}$ for 
  $z\in\mathcal{S}_{j}$. Then each $\varphi_{j}$ is non-negative and bounded by $M^{2}$ on $\mathcal{S}$. Moreover, $\varphi$ is 
  subharmonic on $\mathcal{S}$ and $(\varphi_{j}(z))_{z\bar{z}}\geq 1/2$ for $z\in\mathcal{S}_{j}$. Since the $\varphi_{j}$'s have 
  disjoint support, it follows that  $\varphi:=\sum_{j\in\mathbb{Z}}\varphi_{j}$  is a smooth, bounded function on 
  $\mathcal{S}$ with
  $\varphi_{z\bar{z}}\geq 1/2$ on $\mathcal{S}$.
\end{proof}

\begin{remark}
  Examples of domains, satisfying (i) but not (ii) of Definition \ref{D:largedisks}, for which $\dbar_{0}$ has closed range may be easily constructed using sets $\mathcal{S}$ described above. Let $\mathcal{S}$ be such a set with the additional property $$\eta_{2j+1}(x)-\eta_{2j}(x)>\kappa_{1}\qquad\text{for}\sjump 2\leq|x|\leq 3$$ for some $\kappa_{1}>0$. Let $\Omega_{\mathcal{S}}$ be a domain with smooth boundary satisfying
  \begin{itemize}
    \item[(a)] $\Omega_{\mathcal{S}}\cap\{z\in\mathbb{C}:|\re(z)|>2\}=\mathcal{S}\cap\{z\in\mathbb{C}:|\re(z)|>2\}$,
    \item[(b)] $\{z\in\mathbb{C}:|\re(z)|<1\}\subset\Omega_{\mathcal{S}}$.
  \end{itemize}
  To show that $\dbar_{0}$ for $\Omega_{\mathcal{S}}$ has closed range, a bounded function $\psi\in\mathcal{C}^{2}(\Omega_{\mathcal{S}})$ with $\psi_{z\bar{z}}\geq B$ on $\Omega_{\mathcal{S}}$ for some $B>0$ may be constructed as follows:
\begin{itemize}  
   \item[--] Let $\varphi$ be the function provided by Lemma \ref{L:stripweight} and $\chi\in\mathcal{C}^{\infty}(\mathbb{R})$ 
     with $\chi(x)=0$ for $|x|\leq 2$ and $\chi(x)=1$ for $|x|>3$. Then $\chi\cdot\varphi$ is bounded and $(\chi\cdot\varphi)_{z\bar{z}}\geq 1/2$ on $\Omega_{\mathcal{S}}\cap\{z\in\mathbb{C}:|\re(z)|>3\}$.
   \item[--]  Let $\phi$ be a function as constructed  in Lemma \ref{L:philatticeconstruction}  for the set $\{z\in\mathbb{C}: |\re(z)|<2\}$ with $w_{\ell,k}^{*}\in\{z\in\mathbb{C}\setminus\mathcal{S}: 2<|\re(z)|<3\}$. Then $\phi$ is bounded and strictly subharmonic on $\Omega_{\mathcal{S}}$. In particular, there is a constant $b>0$ such that $\phi_{z\bar{z}}\geq b$ on 
   $\Omega_{\mathcal{S}}\cap\{z\in\mathbb{C}:|\re(z)|\leq 3\}$.
 \end{itemize}
 Then, for sufficiently large $K>0$, $\psi:=\chi\cdot\varphi+K\cdot\phi$ satisfies $\psi_{z\bar{z}}\geq B$ for some $B>0$ on $\Omega_{\mathcal{S}}$.
\end{remark}

\medskip

\subsection{Dimension $n>1$}
An argument analogous to the one given in the proof of Lemma \ref{L:noarblargedisks} yields that arbitrarily large poly-discs (of dimension $n$) are an obstruction to $\dbar_{0}$ having closed range on $L^{2}_{0,1}(\Omega)$ for $\Omega\subset\mathbb{C}^{n}$. The example given in Lemma \ref{L:tube} however shows that this is not a necessary condition.

  \begin{lemma}\label{L:tube}
     Let $D\subset\mathbb{C}$ be a domain which contains arbitrarily large discs.
     Let $m\in\mathbb{N}$ and set 
     $$\Omega=\left\{(z,w)\in D\times\mathbb{C}^{m}:|w|^{2}=|w_{1}|^{2}+\dots+|w_{m}|^{2}<1\right\}.$$  
     Then $\dbar_{0}$ does not have closed range on $L^{2}_{0,1}(\Omega)$, i.e., there is no constant $C>0$ such that
     \begin{align}\label{E:claimHartogs}
       \|u\|\leq C\|\dbar_{0}u\|\quad\sjump\forall\sjump u\in\sd(\dbar_{0})\cap\overline{\sr(\dbarstar_{1})}.
     \end{align}
   \end{lemma}
   
   \begin{proof}
       Let $r(z,w)=|w|^{2}-1$ and  $\alpha_{1}\in\mathcal{C}_{c}^{\infty}(D)$.
       Since $r_z(z,w)=0$ when $(z,w)\in b\Omega$, it then follows that the form $\alpha:=\alpha_{1}d\bar{z}$ belongs to the domain of $\dbarstar_{1}$ of $\Omega$. 
       Set
       \begin{align*}
         u(z):=\dbarstar_{1}\alpha=-\frac{\partial\alpha_{1}}{\partial z}(z)
         =\dbarstar\alpha_{1}(z),
       \end{align*}
       hence $u\in\mathcal{C}_{c}^{\infty}(D)\cap\sr(\dbarstar_{1})\cap\sd(\dbar_{0})$ 
       (for the operators attached to  both $\Omega$ \underline{and} $D$). 
       If \eqref{E:claimHartogs} were to hold then
       \begin{align}\label{E:claimonspecialu}
         \|u\|_{L^{2}(\Omega)}\leq C\|\dbar_{0}u\|_{L^{2}_{0,1}(\Omega)}
       \end{align}
       must hold for all functions $u$ as described above.
       However,
       \begin{align*}
         \|u\|_{L^{2}(\Omega)}^{2}&=\int_{D}|u(z)|^{2}\left(\int_{|w|^{2}<1}1\;dV(w)\right)\;dV(z)\\
         &=c_{m}\int_{D}|u(z)|^{2}dV(z).
       \end{align*}
      Hence, \eqref{E:claimonspecialu} becomes
      \begin{align*}
        \|u\|_{L^{2}(D)}\leq  \tilde C\|\dbar_{0}u\|_{L^{2}_{0,1}(D)}\qquad\sjump\forall\sjump 
        u\in\mathcal{C}_{c}^{\infty}(D)\cap\sd(\dbar_{0})\cap\sr(\dbarstar_{1}),
      \end{align*}
     which contradicts Lemma \ref{L:noarblargedisks}.
    \end{proof}
    \begin{remark}
      It follows from Corollary \ref{C:bounded} (with $\phi=|w|^{2}$) and Lemma  
      \ref{L:closedrangeapprox}
      that $\dbar_q$ for $1\leq q\leq m$ has closed range on $L^{2}_{0,q+1}(\Omega)$.
    \end{remark}

\bibliographystyle{acm}
\bibliography{HerMcN}

\begin{thebibliography}{1}

\bibitem{ChakrabartiShaw11}
{\sc Chakrabarti, D., and Shaw, M.-C.}
\newblock The {C}auchy-{R}iemann equations on product domains.
\newblock {\em Math. Ann. 349}, 4 (2011), 977--998.

\bibitem{Herbig07}
{\sc Herbig, A.-K.}
\newblock A sufficient condition for subellipticity of the
  {$\overline\partial$}-{N}eumann operator.
\newblock {\em J. Funct. Anal. 242}, 2 (2007), 337--362.

\bibitem{Ho95}
{\sc Ho, L.-H.}
\newblock Closed range property of {$\overline\partial$} on nonpseudoconvex
  domains.
\newblock {\em Illinois J. Math. 39}, 1 (1995), 86--97.

\bibitem{Hormander65}
{\sc H{\"o}rmander, L.}
\newblock {$L^{2}$} estimates and existence theorems for the {$\bar \partial
  $}\ operator.
\newblock {\em Acta Math. 113\/} (1965), 89--152.

\bibitem{Hormander90}
{\sc H{\"o}rmander, L.}
\newblock {\em An introduction to complex analysis in several variables},
  third~ed., vol.~7 of {\em North-Holland Mathematical Library}.
\newblock North-Holland Publishing Co., Amsterdam, 1990.

\bibitem{LiShaw13}
{\sc Li, X., and Shaw, M.-C.}
\newblock The {$\overline{\partial}$}-equation on an annulus with mixed
  boundary conditions.
\newblock {\em Bull. Inst. Math. Acad. Sin. (N.S.) 8}, 3 (2013), 399--411.

\bibitem{McNeal02}
{\sc McNeal, J.~D.}
\newblock A sufficient condition for compactness of the
  {$\overline\partial$}-{N}eumann operator.
\newblock {\em J. Funct. Anal. 195}, 1 (2002), 190--205.

\bibitem{range86}
{\sc Range, R.~M.}
\newblock {\em Holomorphic functions and integral representations in several
  complex variables}, vol.~108 of {\em Graduate Texts in Mathematics}.
\newblock Springer-Verlag, New York, 1986.

\bibitem{Shaw10}
{\sc Shaw, M.-C.}
\newblock The closed range property for {$\overline\partial$} on domains with
  pseudoconcave boundary.
\newblock In {\em Complex analysis}, Trends Math. Birkh\"auser/Springer Basel
  AG, Basel, 2010, pp.~307--320.

\end{thebibliography}

\end{document}